\documentclass[11pt,a4paper]{article}

\usepackage{a4wide}
\usepackage{graphicx}
\usepackage{latexsym}
\usepackage{epsfig}
\usepackage{amssymb}
\usepackage{amstext}
\usepackage{amsgen}
\usepackage{amsxtra}
\usepackage{amsgen}
\usepackage{amsthm}
\usepackage{xcolor}

\newtheorem{thm}{Theorem}[section]

\newtheorem{lemma}[thm]{Lemma}

\newtheorem{rem}[thm]{Remark}

\theoremstyle{definition}
\newtheorem{example}[thm]{Example}

\numberwithin{equation}{section}

\newcommand{\R}{\mathbb{R}}
\newcommand{\RR}{\mathbb{R}}

\newcommand{\um}{\hat u_\text{\tiny{MAP}}}
\newcommand{\phm}{\hat p_\text{\tiny{MAP}}}
\newcommand{\phc}{\hat p_\text{\tiny{CM}}}
\newcommand{\qhm}{\hat q_\text{\tiny{MAP}}}
\newcommand{\qhc}{\hat q_\text{\tiny{CM}}}
\newcommand{\uc}{\hat u_{\text{\tiny{CM}}}}

\DeclareMathOperator*{\argmin}{argmin \,}
\DeclareMathOperator*{\argmax}{argmax \,}

\newcommand{\dd}{\text{d}}

\begin{document}
\title{Bregman Cost for Non-Gaussian Noise}
\author{Martin Burger \thanks{Institut f\"ur Numerische und Angewandte Mathematik, Westf\"alische Wilhelms-Universit\"at M\"unster, Einsteinstr. 62, 48149 M\"unster, Germany (martin.burger@wwu.de). The work of this author was supported by by ERC via Grant EU FP 7—ERC Consolidator Grant 615216 LifeInverse}
\and Yiqiu Dong \thanks{Department of Applied Mathematics and Computer Science, Technical University of Denmark, 2800 Kgs. Lyngby, Denmark (yido@dtu.dk). The work of this author was supported by Advanced Grant 291405 from the European Research Council.}
\and Federica Sciacchitano \thanks{Department of Applied Mathematics and Computer Science, Technical University of Denmark, 2800 Kgs. Lyngby, Denmark (feds@dtu.dk)}}

\date{}

\maketitle
\textbf{\large{Abstract}}
One of the tasks of the Bayesian inverse problem is to find a good estimate based on the posterior probability density. The most common point estimators are the   conditional mean (CM) and maximum a posteriori (MAP) estimates, which correspond to the mean and the mode of the posterior, respectively. From a theoretical point of view it has been argued that the MAP estimate is only in an asymptotic sense a  Bayes estimator for the uniform cost function, while the CM estimate is a Bayes estimator for the means squared cost function. Recently, it has been proven that the MAP estimate is a proper Bayes estimator for the Bregman cost if the image is corrupted by Gaussian noise. In this work we extend this result to other noise models with log-concave likelihood density, by introducing two related Bregman cost functions for which the CM and the MAP estimates are proper Bayes estimators. Moreover, we also prove that the CM estimate outperforms the MAP estimate, when the error is measured in a certain Bregman distance, a result previously unknown also in the case of additive Gaussian noise.

\section{Introduction}
%The focus of image restoration consists of recovering the original sharp and noise-free image knowing the blur operator $K$ and the noisy image $f$. 
Bayesian estimation theory deals with the determination of the best estimate of an unknown vector from a related observation. 
In particular, here, we are interested in recovering the original unknown $u\in \RR^n$ from the knowledge of an indirect measurement $f\in \RR^m$ and  the following forward degradation model
\begin{equation}\label{eq:forward}
f=Ku \odot \eta,
\end{equation}
with $K\in \RR^{m\times n}$ being the linear forward operator and
$\eta \in \RR^m$ represents the noise, where $\odot$ might be a multiplication or a sum or even a more complicated operation.  In literature, for simplicity most of the work focused on the additive white Gaussian noise, i.e. when $\eta$ follows the Gaussian distribution, $\mathcal{N}(0,\Sigma)$, with mean $0$ and covariance matrix $\Sigma$ (often a multiple of the identity), see for instance \cite{rof,weiss,Tony2005bookJack}. However, in many real applications, the noise can be much more complicated than additive white Gaussian noise. For example, it might be impulsive \cite{nik, bi:sci, yang}, signal dependent \cite{figue, bi:fo, bi:gu, bi:le}, multiplicative \cite{aa, dz,rlo}, or even mixed \cite{bi:ca,bi:lu, bi:la}. In this paper, we consider a general case, where the noise is just an unknown
realization of a known random noise process. 

Due to the ill-posedness of the inverse problem defined in \eqref{eq:forward}, the simple matrix inversion does not lead to any meaningful solution.
Thus, to recover the original image we can employ the Bayesian approach that combines prior information on $u$ and the forward model in \eqref{eq:forward} to reconstruct the image. By the Bayes' rule \cite{gri}, we have
$$p(u|f)=\frac{p(f|u)p(u)}{p(f)},$$
where $p(u|f)$ is called posterior density and represents the conditional probability density of $u$ given the noisy image $f$, $p(f|u)$ is the likelihood density and it encodes the likelihood that the data $f$ is due to the image $u$, $p(u)$ is the prior density and it describes the properties of the image that we would like to recover, and $p(f)$ is a normalization factor independent of $u$. 
The Bayesian inference deals with extracting the useful information from the posterior to recover the sharp and noise-free image. In particular, it intends to find the best estimate by using the probability of the unknown $u$ given the known observations $f$ and to quantify the uncertainty of the estimate.

A common prior density is the so-called log-concave Gibbs distribution \cite{bov} (using $\propto$ for equality up to a normalization factor)
$$p(u) \propto \exp(-\alpha R(u)),$$
where $R:\RR^n \rightarrow \RR$ is a convex functional and $\alpha>0$ is a scaling parameter, which is often called as regularization parameter. Two popular examples of Gibbs distribution are the Tikhonov regularization and the total variation (TV), see \cite{bi:se} and the references therein. 

The likelihood density depends on the forward degradation model and the noise model. The most common way to write is 
$$p(f|u) \propto \exp(-E(u; K, f)),$$
where $E$ is called data fidelity term and it represents the good fit with the data based on the forward model. In this paper, we assume that $E$ is convex with respect to $u$. For instance, if the original image is corrupted by Poisson noise \cite{figue, bi:fo, bi:le}, based on Poisson distribution $E$ can be written as 
\begin{equation}\label{eq:po}
E(u; K, f) = \sum_{i=1}^{n} \left[ (Ku)_i- f_i \log (Ku)_i + \log(f_i!)\right]
\end{equation}
with $Ku\geq0$.

Thus, based on the Bayes' rule the posterior can be rewritten as 
\begin{equation}\label{post}
	p(u|f) \propto \exp(-E(u; K, f) - \alpha R(u)),
\end{equation}
up to some terms independent of $u$. Now, the question is how to obtain a suitable estimate of the unknown by using the information in the posterior \eqref{post}. 
In the Bayesian inversion approach (cf. \cite{bi:ka,stuart}), there are two popular estimates: the maximum a posteriori (MAP) estimate, $\um$, and the conditional mean (CM) estimate, $\uc$. They are defined as (cf. \cite{dashti,bi:he} for infinite-dimensional versions)
\begin{equation}\label{eq:um_uc}
\begin{aligned}
\um&=\argmax_{u \in \mathbb{R}^n} p(u|f)=\argmin_{u \in \mathbb{R}^n} E(u; K, f) + \alpha R(u),\\ 
\uc&= \mathbb{E}[u|f]=\int u\, p(u|f) \ \dd u.\end{aligned}
\end{equation}
According to their definitions, the MAP estimate corresponds to find the mode of the posterior, while the CM estimate corresponds to compute the expected value of the posterior. Of course the quality of the different estimates as representations of the posterior distribution is an important question.
%One of main topics in the Bayesian inference is to study the best estimate. Although there are a lot of arguements on which estimate is the best, there are a few points that we should always take into account.

Computing the MAP estimate leads to solving a high-dimensional optimization problem and the CM estimate leads to solving a high-dimensional integration problem. From the numerical point of view, the MAP estimate can be computed rather efficiently, see for instance \cite{bi:se}, while the CM estimate requires to solve a much harder and more time-consuming integration problem. To calculate the CM estimate the classical techniques of numerical quadrature seem prohibitive in high-dimension, hence Monte Carlo methods  or special sparsity techniques \cite{schwabstuart} have to be employed. Further, drawing samples from the posterior is often not straight-forward, so the Markov chain Monte Carlo (MCMC) techniques need be used, for an overview see \cite{bi:ka}. Although there are computational challenges to calculate the CM estimate, it has many theoretical benefits. Comparing with the MAP estimate, the CM estimate is a more intuitive choice, since it represents the average of the samples. Moreover, from the theoretical point of view, the CM estimate is the Bayes estimator for the mean squared error cost, while the MAP estimate is only asymptotically the Bayes estimator for the uniform cost function. Recently, in \cite{burlu}, in the Gaussian noise case, it has been shown that the MAP estimate is a Bayes estimator for the cost function given by an $L^2$ term and a Bregman distance. More details will be given in Section \ref{Sec:MAPCM}. In \cite{bi:he}, the results in \cite{burlu} are extended to the infinite-dimensional setting. 

The main novelty of this paper is to study from the Bayesian cost point of view the CM and MAP estimator in non-Gaussian noise cases. In this paper, we provide several cost functions for which the MAP and CM estimate are Bayes estimator not only under the Gaussian noise model, but under more general noise models. The only assumption that we need is the convexity and Lipschitz-continuity of $E$. In addition, we will show that under some assumptions the CM estimate outperforms the MAP estimate in an appropriate error measure. 
 
The remainder of the paper is organized as follows. In Section \ref{Sec:MAPCM}  we analyse the difference between the MAP estimate and the CM estimate, and give an overview of the Bayesian approach. In Section \ref{Sec:MAP} we provide a cost function for which the MAP estimate is Bayes estimator. In Section \ref{Sec:CM} we study the optimality condition for the CM estimate, and give some suitable cost functions for  the CM estimate. In Section \ref{Sec:comparison}, we compare the two estimates by proving that the CM estimate outperforms the MAP estimate when the error is measured using a cost function that depends on the Bregman distance.  The conclusion are drawn in Section \ref{Sec:con}.

%%%%%%%%%%%%%%%%%%%%%%%%%%%
\section{Review of Bayes Cost Formalism}\label{Sec:MAPCM} 

%Based on the definition in \eqref{eq:um_uc}, it is clear that the MAP estimate only shows the mode of the posterior, thus some information may be lost. Especially, the MAP estimate only depends on the location of the mode, but not on how much probability this mode contains, see \cite{bi:ru} for more details. While in high dimension the probability mass in the posterior distribution is also important for the inference in addition to the location of the mode.
%Furthermore, in \cite{bi:lo} and \cite{bi:ab} it has been shown under some assumptions that for the Gaussian noise and Poisson noise the reconstructed images given by the CM estimate do not have any staircasing artefacts, which are typically presented in the MAP estimate, when it is used the TV prior.
%
%
%From the numerical point of view, the MAP estimate seems to be preferred to the CM estimate since the first one requires to solve a optimization problem while the second one deals with an integration problem, which is quite complicate in high-dimension. In this section, we point out the differences between the CM and MAP estimates by focusing on some theoretical arguments. 
%

One main focus of Bayesian technique is the determination of the best estimate of an unknown data. For instance, in the inverse problem defined in \eqref{eq:forward} it would be to find the best estimation of the original image $u$, which is corrupted by blur and noise. 
%As we will define later, the best estimate is the one that minimizes a cost function and that involves the use of probability of the clean image knowing the blurred and noise image.
The Bayesian estimation of $u$ from the given noisy image $f$ relies on the minimization of a Bayes cost, which is defined as follows
$$\begin{aligned} 
BC_C(\hat u):=& \mathbb{E}[C(u,\hat u)]\\
=&\int \int C(u,\hat u)p(u,f) \ \dd u\  \dd f \\ 
=& \int \int C(u,\hat u)p(u|f)\  \dd u \ p(f) \  \dd f,
\end{aligned}$$
%=& \int \int C(u,\hat u(f))p(f|u) \dd f p(u) \dd u\\
where $C:\mathbb{R}^n\times \mathbb{R}^n \to \mathbb{R}$ is a cost functional measuring the distance between $u$ and $\hat u$. A Bayes estimator $\hat u_C$ is the estimator minimizing the Bayes cost function $BC_C(\hat u)$, that is
\begin{equation}\label{eq:be}
\hat u_C:= \argmin_{\hat u} BC_C(\hat u).
\end{equation}
Since $\hat{u}$ only depends on $f$ and the marginal density $p(f)$ is non-negative, the Bayes estimator can be also computed by
$$\hat u_C:= \argmin_{\hat u} \int C(u,\hat u(f))p(u|f)\ \dd u.$$
Therefore, the Bayes estimator is always corresponding to certain cost functions, and it's very important to find a suitable cost function. 

One of the most common choice for the cost function is the mean squared error, i.e.
\begin{equation}\label{ep:mse}
C(u,\hat u)=\|u-\hat u\|_2^2,
\end{equation}
and the conditional mean estimate $\uc$ is the corresponding Bayes estimator. 
% $$\begin{aligned}  \int \|u-\hat u(f)\|_2^2 p(u|f) \dd u =& \int \|u- \uc\|_2^2 p(u|f) \dd u +\int \|\hat u(f)-\uc\|_2^2 p(u|f) \dd u \\  &\; -2 \int \langle u- \uc, % \hat u(f)-\uc\rangle p(u|f) \dd u \\
% =& \int \|u- \uc\|_2^2 p(u|f) \dd u +\int \|\hat u(f)-\uc\|_2^2 p(u|f) \dd u
% \end{aligned}$$
% where the last equality holds by definition of CM estimate. Thus, since the first term is independent of $\hat u$, we have that the minimizer is given % by $\hat u=\uc$. 
Another popular choice for the cost function is the uniform cost, i.e.
$$C(u,\hat u)=
\begin{cases}
0, &\quad |u_k -\hat u_k|<\epsilon \quad \text{for} \; 1\leq k \leq n, \\
1, &\quad \text{otherwise},
\end{cases} $$
where $\epsilon>0$ is a small constant. It turns out that the MAP estimate $\um$ is an asymptotic Bayes estimator for this cost function.
%$$ \int_{|u_k -\hat u_k|>\epsilon}  p(u|f) \dd u =
%1- \prod_{k=1}^n \int_{\hat u_k-\epsilon}^{\hat u_k+\epsilon} p(u|f) \dd u_k
%\approx 1- (2\epsilon)^n p(\hat u|f),
%$$
%where the last approximation holds by the mean value theorem. Thus, minimizing the uniform cost is asymptomatically equivalent to computing the MAP estimate. 

Although it seems intuitively optimal to use the squared Euclidean norm, i.e. variance, as a Bayes cost functional, there is no real justification in high-dimensional version with a non-Gaussian prior. Assume e.g. that $R$ is some power of a norm different from the Euclidean one (e.g. the popular $\ell^1$- or total variation norm), then effectively $R$ induces the relevant Banach space geometry on $\RR^n$. For increasing $n$ this geometry is very different from the Euclidean one for large $n$ and in a limit $n \rightarrow \infty$ one might even converge to a Banach space setting where no equivalent of the Euclidean norm exists, hence the standard variance becomes questionable. Hence, in such a setting different cost functionals better adapted to the structure of induced by $R$ shall be benefitial, both for characterizing the MAP and CM estimate.

Recently, in \cite{burlu} it has been shown that the MAP estimate is a Bayes estimator for 
$$C(u,\hat u)=\|K (\hat u-u)\|_2^2 +2\alpha D_R^q(\hat u, u),$$
where $K$ is the blurring operator, $\alpha>0$ is a regularization parameter, and $D_R^{q}(\hat{u}, u)$ represents the Bregman distance between $\hat{u}$ and $u$ for a convex regularization functional $R$ and a subgradient $q \in \partial R(u)$, which is defined as
$$D_R^q(\hat u, u)= R(\hat u) -R(u) - \langle q, \hat u -u\rangle. $$
If $R$ is Fr\'echet differentiable in $u$, then the subgradient $q$ corresponds to the standard Fr\'echet derivative $R'$. In this paper, we refer to the Bregman distance by omitting $q$, i.e. $D_R(\hat u, u)$. The Bregman distance has been introduced in \cite{bi:be} and it is a very useful tool in image processing, see for instance \cite{bi:bu, bi:bur, bi:go}. Since it is not symmetric and the triangle inequality does not hold, it is not a distance in the mathematical sense. But some nice properties hold, such as
 \begin{itemize}
 \item $D_R(\hat u, u)\geq 0$.
 \item If $R$ is strictly convex, $D_R(\hat u, u)=0$ implies $\hat u=u$.
 \item $D_R(\hat u, u)$ is convex in $\hat u$.
 \end{itemize}

In \cite{burlu}, all the comparison of the MAP and CM esitmates are under the additive Gaussian noise model, i.e., $E(u; K, f)=\|Ku-f\|_{2}^{2}$, and it has been proven that the MAP estimate is a proper Bayes estimator in this case. In this paper, we will discuss and compare the MAP and CM esitmates under more general noise models. The main assumption we need is that the considered noise model leads to a convex data fitting term $E(u; K, f)$. More precisely we shall assume that the functionals $u \mapsto E(u;K,f)$ and $R$ are nonnegative, convex and Lipschitz-continuous on $\R^n$ without further notice. Moreover, we assume that the posterior is well specified by \eqref{post}, i.e. 
$$  \int_{\RR^n} \exp(-E(u; K, f) - \alpha R(u)) ~du < \infty . $$

%%%%%%%%%%%%%%%%%%%%%%%%%%%%%%%%%%%%% 
 \section{Cost Function for the MAP Estimate}\label{Sec:MAP}
 
To propose a cost function for the MAP estimate, we first show that the posterior distribution in \eqref{post} can be rewritten in a MAP-centred form by using the optimality condition of the MAP estimate. 

Since the MAP estimate $\um \in \mathbb{R}^n$ is a maximizer of the posterior defined in \eqref{post}, it satisfies the optimality condition 
\begin{equation}\label{optcond} 
K^\top \qhm + \alpha \phm = 0,
\end{equation}
where $\qhm \in \partial_u E(\um; K, f)$ and $\phm \in \partial R(\um)$. Then, we can obtain the following result.
%Adding some terms independent of $u$ and using the optimality condition \eqref{optcond},we can rewrite the posterior in a MAP-centred way.

 \begin{lemma}\label{lemmapost}
 The posterior in \eqref{post} can be rewritten in a MAP-centred form
\begin{equation}\label{newpost}
	p(u|f)\propto 	\exp(-D_E^{\qhm}(u, \um)  - \alpha D_R^{\phm}(u,\um) ),
\end{equation}
where $D_E^{\qhm}(u,\um)$ (resp. $ D_R^{\phm}(u,\um)$) indicates the Bregman distance between $u$ and $\um$.
\end{lemma}

\begin{proof}
%Let's prove that the posterior probability distribution in \eqref{post} is equivalent to the posterior probability distribution MAP-centred in \eqref{newpost} up to some terms independent of $u$. 

First, we would like to point out that we are allowed to ignore the terms independent of $u$, since the minimization problem in \eqref{eq:um_uc} is only on $u$. Then, based on the definition of Bregman distance, we have
\begin{equation}\label{pr} 
\begin{aligned}
\exp(-&D_E(u, \um)  - \alpha D_R(u,\um) ) \\
&\propto  \exp(- E(u; K, f)+\langle \qhm, Ku - K \um \rangle  - \alpha R(u) + \alpha \langle \phm, u- \um \rangle) \\
&= \exp(- E(u; K, f) - \alpha R(u) +  \langle  K^\top \qhm + \alpha\phm, u- \um \rangle) \\
&=\exp(- E(u; K, f) - \alpha R(u) )\\
&\propto p(u|f),
\end{aligned}
\end{equation} 
where $\qhm\in\partial E(\um; K,f)$, $\phm\in\partial R(\um)$, and the last equality follows directly from the optimality condition in \eqref{optcond}.
%Using the definition of Bregman distance, we have
%$$\begin{aligned}
%D_R^{\hat p_{MAP}}(u,\hat u_{MAP})&=R(u)-R(\hat u_{MAP})- \langle \hat p_{MAP}, u-\hat u_{MAP} \rangle \\
%&= R(u)-R(\hat u_{MAP})+\frac{1}{\alpha} \langle K^*\partial_u E(K \hat u_{MAP},f), u-\hat u_{MAP} \rangle\\
%&= R(u)-R(\hat u_{MAP})+\frac{1}{\alpha} \langle \partial_u E(K \hat u_{MAP},f), K(u-\hat u_{MAP} )\rangle,
%\end{aligned}$$
%where the second last equality follows from the optimality condition in \eqref{optcond}. 
%Substituting the above expression in \eqref{pr}, we get
%$$ \partial_u E(K \hat u_{MAP},f)(Ku - K\hat u_{MAP})  - \alpha R(u)+ \alpha R(\hat u_{MAP})- \langle \partial_u E(K \hat u_{MAP},f), K(u-\hat u_{MAP} )\rangle = - \alpha R(u),$$
%which is always true up to  term independent of $u$.
\end{proof}

Now, we suggest a cost function for which the MAP estimate is a Bayes estimator. For simplicity of notation we omit writing the subgradient as a superscript in the Bregman distance, since due to the Lipschitz-continuity of the involved functionals the subgradient contains a single element almost everywhere.

\begin{thm} \label{costmap}
Under the decay assumption 
\begin{equation}\label{eq:ass}
\lim_{r\to \infty} \int_{\partial\mathcal{B}_r(0)} p(u|f) \dd s=0,
\end{equation} 
for the posterior defined by \eqref{post}, the MAP estimate minimizes the Bayes cost with cost functional
\begin{equation}\label{cost}
 C(\hat u, u) = D_E(\hat u,u) + \alpha  D_R(\hat u,u).
\end{equation}
\end{thm}

%\begin{lemma}\label{remmap}
%The cost function in \eqref{cost} can be rewritten  in a MAP-centered way
% \begin{equation}\label{costnew}
% C(\hat u, u) = D_F(\hat u, \um)+ \alpha  D_R^{\phm}(\hat u,\um) + \partial_u(\log p(u|f))(\hat u- u).
%\end{equation}
%\end{lemma}
%To prove the Lemma we will use the following elementary identity for the Bregman distance 
%$$D_R(\hat u,u)=D_R(\hat u, \um)+ D_R(\um, u)+ \langle \hat p_\text{MAP}- R'(u), \hat u-\um \rangle.$$
%For more details see the Appendix or cite someone.

\begin{proof}
Based on the definition in \eqref{eq:be}, $\um$ is a Bayes estimator for the cost function $C(\hat u,u)$, if it satisfaies
$$\um \in \arg \min_{\hat u} \int C( \hat  u, u)p(u|f)\ \dd u.$$ 
%Without loss of generality, we do not consider the null-set where  $E(Ku,f)$ and $J(u)$ are not Fr\'echet differentiable.
Using the elementary identity for the Bregman distance 
$$D_R(\hat u,u)=D_R(\hat u, \um)+ D_R(\um, u)+ \langle \phm- p, \hat u-\um \rangle \  \text{ with} \;p\in \partial R(u)$$
 and ignoring the terms independent on $\hat u$, we have
$$ \begin{aligned} \label{mapcost}
 C(\hat u, u)
 &=D_E(\hat u, \um)+\alpha  D_R(\hat u, \um)\\
 &\quad+\langle \qhm, K\hat u-K\um\rangle +\alpha \langle \phm, \hat u-\um \rangle  \\  &\quad - \langle q, K\hat u-Ku\rangle -\alpha \langle  p, \hat u-\um \rangle, \end{aligned}$$
where $\qhm\in\partial E(\um; K,f)$, $\phm\in\partial R(\um)$. It is obvious that $D_E(\hat u, \um)+\alpha D_R(\hat u, \um)$ reaches to minimal at $\hat u= \um$. Due to the optimality condition in \eqref{optcond}, the two terms in the second line vanish. Therefore, if we can show
\[
 \int\left( \langle q, K\hat u-Ku\rangle +\alpha \langle  p, \hat u-\um \rangle\right)p(u|f)\ \dd u=0,
\] 
we have proved that the MAP estimate is a Bayes estimator for $C(\hat u,u)$. 

Up to some terms independent on $\hat u$, we have
$$ \begin{aligned} \langle q, K\hat u-Ku\rangle -\alpha \langle  p, \hat u-\um \rangle&=\langle q,K\hat u\rangle  +\alpha \langle  p, \hat u \rangle \\
&=\langle K^\top q+ \alpha  p, \hat u \rangle \\
&=\langle \nabla_u \log p(u|f), \hat u \rangle.
\end{aligned}$$
Using the logarithmic derivative $\nabla_u p(u|f)=\left(\nabla_u \log p(u|f)\right) p(u|f)$, we obtain
$$\begin{aligned} 
 \int\left( \langle q, K\hat u-Ku\rangle +\alpha \langle  p, \hat u-\um \rangle\right)p(u|f)\ \dd u&=\int \langle \nabla_u \log p(u|f),  \hat u\rangle p(u|f)\ \dd u\\ &= \biggl\langle \int  \nabla_u  \log p(u|f) p(u|f)\ \dd u,  \hat u\biggr\rangle \\
&= \biggl\langle \int  \nabla_u  p(u|f) \ \dd u,  \hat u\biggr\rangle.
\end{aligned}$$
By the assumption \eqref{eq:ass}, we have
$$\begin{aligned} 
\biggl\|\int  \nabla_u  p(u|f)\  \dd u\biggr\| &=\lim_{r\to \infty}\biggl\|\int_{\mathcal{B}_r(0)}  \nabla_u ( p(u|f) ) \dd u\biggr\| \\
&=\lim_{r\to \infty}\biggl\|\int_{\partial \mathcal{B}_r(0)}   p(u|f) \cdot \mathbf{n}\; \dd s \biggr\| \\
&\leq \lim_{r\to \infty}\int_{\partial \mathcal{B}_r(0)}  p(u|f)  \dd s \\
&=0,
\end{aligned}$$
where the second equality follows the Gauss's theorem and $\mathbf{n}$ indicates the outward unit normal to the surface. 

%Thus, since the Bayes cost function in \eqref{cost} reaches the minimum when $\hat u=\um$, we prove the theorem. 
%Writing explicitly the Bregman distance we have
%$$\begin{aligned}
%- \partial E(K u,f)(K\hat u-Ku) &+ \alpha D_R^{q}(\hat u,u) \\ &= - \partial  E(K u,f)(K\hat u-Ku) + \alpha (R(\hat u)- R(u)- \langle q, \hat u-u \rangle)\\
%&= - \langle K^*\partial  E(K u,f), \hat u-u\rangle + \alpha (R(\hat u)- R(u)- \langle q, \hat u-u \rangle)\\
%&=  \alpha (R(\hat u)- R(u)) -  \langle K^*\partial E(K u,f)+ \alpha q, \hat u-u \rangle\\
%&= \alpha R(\hat u)- \alpha R(u) +  \langle \partial_u ( \log p(u|f) ),  \hat u-u \rangle.
%\end{aligned}$$
%Using the optimality condition in \eqref{optcond} and ignoring the term independent of $\hat u$, we have
%$$\begin{aligned}
%- \partial  E(K u,f)&(K\hat u-Ku) + \alpha D_R^{q}(\hat u,u) \\ &= 
%\alpha R(\hat u) +  \langle \partial_u ( \log p(u|f) ), \hat  u-u \rangle - \langle K^* \partial  E(K \hat u_{M},f) + \alpha \hat p_{M}, \hat u-\hat u_{M} \rangle
%\\ &= 
%\alpha R(\hat u) +  \langle \partial_u ( \log p(u|f) ),  \hat u-u \rangle - \langle  \partial  E(K \hat u_{M},f),  K\hat u- K \hat u_{M} \rangle -\langle \alpha \hat p_{M}, \hat u-\hat u_{M} \rangle \\
%&=   \langle \partial_u ( \log p(u|f) ),  \hat u-u \rangle - \langle  \partial  E(K \hat u_{M},f),  K\hat u- K \hat u_{M} \rangle + \alpha  D_R^{\hat p_{M}}(\hat u,\hat u_{M}),
%\end{aligned}$$
%up to a term independent of $\hat u$. 
%
%Hence, we proved that up to terms independent of $\hat u$, \eqref{cost} and \eqref{costnew} are equivalent.
\end{proof}

We finally mention that the cost function defined in \cite{burlu} for the case of additive Gaussian noise is indeed a special case of \eqref{mapcost}.
\section{Cost Functions for the CM Estimate}\label{Sec:CM}

The CM estimate has been proved as a Bayes estimator for the mean squared error \eqref{ep:mse} independently on the noise model. In this section, we will show that there exist other cost functions for which the CM estimate is Bayes estimator. The proposed cost functions are the sum of the Bregman distances of any convex functions, and they are independent on the prior probability density and the noise model as well.

\subsection{Optimality condition of the CM estimate}

%Here, we compare the CM and the MAP estimate by using the optimality condition.
%Since the MAP estimate is the maximizer of the log-likelihood, we have that the MAP estimate satisfied the optimality condition, i.e.,
%$$K^* \qhm + \alpha \phm = 0, \quad \qhm \in \partial_u E(K \um,f)\text{and} \quad \phm \in \partial R(\um).$$ 
In \cite[Sect. 3.3]{burlu}, under the additive Gaussian noise model it has been proved that the CM estimate fulfils an optimality condition ``on average'', i.e. with respect to the average gradient $\phc=\mathbb{E}[\partial R(u)]= \int p\; p(u|f)\ \dd u$ with $p\in\partial R(u)$, we have
\[
K^\top(K\uc-f)+\alpha\phc=0.
\]
In the following theory, we prove that under more general noise models the CM estimate does not always fulfill the optimality condition  ``on average'', and with some assumptions on the posterior, this optimality condition can be strictly positive. 

\begin{thm}\label{avpo}
Assume that the decay assumption \eqref{eq:ass} holds, the operator $K$ is positive, that $E$ is differentiable with respect to $u$, and the maps $u \mapsto \partial_{u_i} E(u; K, f)$, $i=1,\ldots,n$ are concave for every $u$. Then we have 
\begin{equation}\label{eq:CM_opt} 
\qhc +\alpha\phc \geq 0,
\end{equation}
where $\qhc \in \partial E(\uc; K, f)$. The equality in \eqref{eq:CM_opt} holds if and only if $q$ is linear or a constant with respect to $u$.
\end{thm}

\begin{proof}
Based on the forward model in \eqref{eq:forward}, the data fidelity term $E$ is composed of $Ku$ and $f$. For convenience, we also use the notation $E(Ku; f)$ instead of $E(u; K, f)$, when the operator $K$ plays an important rule.

According to the linearity of $K$ and the definition of $\uc$, we have
$$K^\top \partial E(K\uc; f)=K^\top \partial E(K \mathbb{E}[u]; f)=K^\top \partial E( \mathbb{E}[K u]; f)$$ 
By Jensen's inequality, if $\phi$ is concave we have $$\phi (\mathbb{E}(x))\geq \mathbb{E}(\phi(x)),$$ 
and the equality holds if and only if $\phi$ is a constant or a linear function.
Thus, by the positivity condition on the operator $K$, we have \begin{equation}\label{ave}
\begin{aligned}
K^\top \partial E(K\uc; f)+\alpha\phc&=K^\top \partial E( \mathbb{E}[K u]; f)+\alpha\phc\\
& \geq  K^\top \mathbb{E}[\partial E(Ku; f)]+  \alpha \int p\;p(u|f)\ \dd u \\
&= \mathbb{E}[K^\top \partial E(Ku; f)]+  \alpha \int p\;p(u|f)\ \dd u \\
&= \int K^\top \partial E(Ku; f)p(u|f)\ \dd u+  \alpha\int p\;p(u|f)\ \dd u \\
&= \int \nabla_u \log p(u|f) p(u|f)\ \dd u =\int \nabla_u p(u|f)\ \dd u =0.
\end{aligned}
\end{equation}
%Using the same argument in \cite{burlu}, we know
%\begin{equation}\label{ave}
%\begin{aligned}
%0&=\int \nabla_u p(u|f)du = \int \nabla_u p(u|f) p(u|f)du \\
%&= \int K^\star \partial E(Ku,f)p(u|f)du+  \alpha\int R'(u)p(u|f)du \\
%&= K^\star \mathbb{E}[\partial E(Ku,f)]+  \alpha \int R'(u)p(u|f)du,
%\end{aligned}
%\end{equation}
%where the last equality holds by linearity of the blur operator.
%By Jensen's inequality, if $\phi$ is convex we have $$\phi (\mathbb{E}(x))\leq \mathbb{E}(\phi(x)).$$ The equality holds if and only if $\phi$ is a constant or a linear function.
%%Thus, $$\frac{1}{\uc}=\frac{1}{\int u p(u|f)du}< \int \frac{1}{u}p(u|f)du.$$
%Therefore, \eqref{ave} can be maximize as follows
%$$0< K^\star \partial E(K \mathbb{E}[u],f)+\alpha\phc=K^\star \partial E(K \uc,f)+\alpha\phc.$$ 
\end{proof}

Note that in this paper $K$ indicates the blurring operator, therefore the positivity condition is always satisfied.

\begin{example}
In this example, we consider the Poisson noise model. The corresponding data fidelity term is given in \eqref{eq:po}, which is convex, so Theorem \ref{avpo} can be apply to it. Based on the definition in \eqref{eq:po}, we have
\[
q=K^{\top} v \quad \mbox{ and }\quad v_{i}=1-\frac{f_{i}}{(Ku)_{i}}+g_{i}\in\mathbb{R}^{n},
\]
where $g\in\partial I_{\{i: (Ku)_{i}\geq0\}}(u)$ and $I_{\{i: (Ku)_{i}\geq0\}}(u)$ denotes the indicator function. Since $q$ is concave with respect to $u$, based on Theorem \ref{avpo} we can conclude that 
$$\qhc+\alpha\phc >0.$$

In \cite{burlu}, under the Gaussian noise model we have 
\[
q=K^{\top}(Ku-f),
\]
which is linear with respect to $u$, Therefore, based on Theorem \ref{avpo} the optimality condition ``on average'' is satisfied. 
\end{example}

%%%%%%%%%%%%%%%%%%%%%%%%%%
\subsection{Cost function for the CM estimate}

The CM estimate has been proved as a Bayes estimator for the mean squared error \eqref{ep:mse} independently on the noise model. In the following theory, we prove that for any linear combinations of Bregman distances as cost functions the CM estimate is also a Bayes estimator.

\begin{thm}\label{costCM}
The CM estimate is a Bayes estimator for any cost function, which is a linear combination of Bregman distances of convex functions, i.e.
\begin{equation}\label{eq:costcm}
C_{CM}(\hat u, u)= \sum_{i=1}^N D_{F_i}^{\hat q_i}(u, \hat u),
\end{equation}
where $F_1, \dots, F_N$ are convex functions with finite expectation under $p(\cdot|f)$ and $\hat q_i \in \partial F_i(\hat u)$. 
\end{thm}

\begin{proof}
Based on the definition of Bayes estimator, we need prove that 
$$\uc \in \arg \min_{\hat u} \int C_{CM}( \hat  u, u)p(u|f)\dd u.$$ 
By expanding the Bregman distance and collecting constant terms independent of $\hat u$, we have
$$ \begin{aligned}
\int  C_{CM}(\hat u, u) p(u|f)\  \dd u&= \int\left( \sum_{i=1}^N  D_{F_i}^{\hat q_i}(u,\hat u)\right) p(u|f)\ \dd u \\ &= 
\int \left[\sum_{i=1}^N  (F_i(u)- F_i(\hat u)- \langle \hat q_i, u-\hat u \rangle ) \right]p(u|f)\ \dd u  \\
 &=\sum_{i=1}^N   (- F_i(\hat u)- \langle \hat q_i, \uc-\hat u\rangle )+ \text{const} \\ 
 &= \sum_{i=1}^N  D_{F_i}^{\hat q_i}(\uc,\hat u) + \text{const}. \end{aligned}$$
 Obviously, it attains the minimal value at $\hat u= \uc$.
\end{proof}

%As in Martin's report, a suitable Bayes cost for the CM estimate is 
%\begin{equation}\label{costCM}
% C_{CM}( \hat  u, u) = E(K u,f)-E(K\hat u,f) - \partial_{ \hat u} E(K \hat u,f)(Ku-K\hat u) + \alpha  D_R^{q}( u, \hat u) 
%\end{equation}
%with $q \in \partial R(\hat u)$.

In \cite{burlu}, it has been shown that the cost function for the CM estimate can be independent on the prior under the additive Gaussian noise model. In Theorem \ref{costCM} we give a much more general result. We have proved that the cost function for the CM estimate can be completely independent on the posterior density, i.e., independent on the noise model and also prior density. However, in Theorem \ref{costmap} we have shown that the cost function for the MAP estimate depends on the likelihood and the prior density. In the end of this section, we list a few examples of cost functions for $\uc$.

\begin{example}
From Theorem \ref{costCM}, we have the following functions that are suitable cost functions for the CM estimate, 
$$ \begin{aligned}
 C^1_{CM}( \hat  u, u) &= E(u; K, f)-E(\hat u; K, f) - \partial E(\hat u; K, f)(Ku-K\hat u)  = D_E(u,\hat u) \\ 
 C^2_{CM}( \hat  u, u) &=   D_R( u, \hat u) \\
  C^3_{CM}( \hat  u, u) &=  D_E(u,\hat u) + \alpha D_R( u, \hat u).
 \end{aligned}$$ 
 % The CM estimate is a Bayes estimator for  $C^i_{CM}( \hat  u, u)$ if the following equality is satisfied
%$$\uc \in \arg \min_{\hat u} \int C^i_{CM}( \hat  u, u)p(u|f)du, \quad \text{ for } i=1,2.$$ From a direct computation we have
%$$ \begin{aligned}\int  C^1_{CM}(\hat u, u) p(u|f)du&= \int [E(K u,f)-E(K\hat u,f) - \partial_{ \hat u} E(K \hat u,f)(Ku-K\hat u)] p(u|f)du \\ &= 
% -E(K\hat u,f) - \partial_{ \hat u} E(K \hat u,f)(K\uc-K\hat u) + \text{const} \\
% &= E(Ku,f) -E(K\hat u,f) - \partial_{ \hat u} E(K \hat u,f)(K\uc-K\hat u) + \text{const} \\ &= D_E(\uc,\hat u) + \text{const} \end{aligned}$$ and
% $$ \begin{aligned}\int  C^2_{CM}(\hat u, u) p(u|f)du&= \int D_R^q(u,\hat u) p(u|f)du \\ &= 
%\int [R(u)- R(\hat u)- \langle q, u-\hat u \rangle ] p(u|f)du  \\
% &= - R(\hat u)- \langle q, \uc-\hat u\rangle  + \text{const} \\ &= D_R(\uc,\hat u) + \text{const}, \end{aligned}$$
% which are obviously minimized at $\hat u= \uc$.
 \end{example}
\begin{rem}
 Note that  $C^3_{CM}( \hat  u, u) $ resembles the form of the cost function for the MAP estimate in \eqref{cost}, but is different since the Bregman distance is not symmetric, except for the case of Gaussian posterior, when MAP and CM estimate coincide. 
 \end{rem}

 \section{Comparison of the CM and MAP Estimates}\label{Sec:comparison}
 
In \cite[Thm. 2]{burlu}, it has been shown that under the additive Gaussian noise model the CM estimate performs better than the MAP estimate when the error is measured in a quadratic distance, $\|L(u-\hat u)\|_2^2$ with any linear operator $L$. But if the error is measured by the Bregman distance $D_R(\hat u, u)$, the MAP estimate outperforms the CM estimate. In this section, based on the results in Theorem \ref{costCM}, we give another comparison result under more general noise models, for instance the Laplacian noise or the Poisson noise. 
 
 \begin{thm}\label{thm:fin}
The CM estimate outperforms the MAP estimate when the error is measured in the Bregman distance $D_R(u, \hat u)$, i.e. 
 $$\mathbb{E}[ D_R(u, \uc)]\leq \mathbb{E}[D_R(u, \um)].$$
 \end{thm}
  
  This inequality directly follows from the fact that $\uc$ is a Bayes estimator for $C^2_{CM}(\hat{u}, u)$. 
 Note that the above result is exactly the opposite of Theorem 2 in \cite{burlu} under the additive Gaussian noise model, but with flipped $\hat u$ and $u$ in the definition of cost function.

\section{Conclusions}\label{Sec:con}
In this paper, based on image restoration problem with the more general noise models instead of additive Gaussian noise, we study the two typical point estimators for the posterior probability density: the conditional mean (CM) estimate and the maximum a posteriori (MAP) estimate. The only assumption that we need is that the considered noise model has to lead to a convex data fidelity term. Based on the Bregman distance, we propose new cost functions for which the MAP and the CM estimate are Bayes estimators. Further, we give a new comparison result on these two estimates. In addition, we give the posterior in a MAP-centred form and study the optimality condition on average of the CM estimate.

\end{document}